\documentclass[12pt]{article}
\usepackage{amsfonts, amsmath, amsthm, amssymb, relsize, scalerel, graphics, tabu}
\usepackage[table]{xcolor}
\usepackage [english]{babel}
\usepackage{cjhebrew}
\usepackage{mathrsfs}
\usepackage{enumerate}
\usepackage [autostyle, english = american]{csquotes}
\usepackage{harpoon}

\newtheorem{theorem}{Theorem}[section]
\newtheorem{proposition}[theorem]{Proposition}
\newtheorem{corollary}[theorem]{Corollary}
\newtheorem{lemma}[theorem]{Lemma}

\theoremstyle{definition}

\newtheorem{definition}[theorem]{Definition}
\newtheorem{remark}[theorem]{Remark}
\newcommand{\ii}{\mathrm{i}}

\title{A B\'{e}koll\`{e}-Bonami Class of Weights for Certain Pseudoconvex Domains}
\author{Zhenghui Huo, Nathan A. Wagner\footnote{Supported by NSF GRF, grant number DGE-1745038}, and Brett D. Wick\footnote{Supported by NSF grants DMS-1800057 and DMS-1560955, as well as ARC DP190100970.}}

\begin{document}

\maketitle

\begin{abstract} We prove the weighted $L^p$ regularity of the ordinary Bergman projection on certain pseudoconvex domains where the weight belongs to an appropriate generalization of the B\'{e}koll\`{e}-Bonami class. The main tools used are estimates on the Bergman kernel obtained by McNeal and B\'{e}koll\`{e}'s original approach of proving a good-lambda inequality. 

\end{abstract}

\section{Introduction}

It is a well-known result of B\'{e}koll\`{e} and Bonami that the Bergman projection $\mathcal{P}$ is bounded on $L^p_\sigma(\mathbb{B}_n)$, where $\mathbb{B}_n$ is the unit ball, if and only if the weight $\sigma$ belongs to the so-called B\'{e}koll\`{e}-Bonami class of weights (see \cite{Bekolle}, \cite{Bekolle2}, \cite{RTW}). These weights are defined by the following Muckenhoupt-type condition:

$$ \sup_{B(w,R); R>1-|w|} \left(\frac{1}{\mu(B(w,R))}\int_{B(w,R)} \sigma \mathop{d\mu}\right)\left(\frac{1}{\mu(B(w,R))}\int_{B(w,R)} \sigma^{-1/(p-1)} \mathop{d\mu}\right)^{p-1}<\infty,$$

\noindent where $\mu$ denotes Lebesgue measure and the balls $B$ are taken in the quasi-metric defined by $$d(w,z)=||w|-|z||+\left|1-\frac{\langle w,z \rangle}{|w||z|}\right|.$$

\noindent The general framework used in B\'{e}koll\`{e}'s paper is singular integral theory: crucial smoothness estimates are obtained on the kernel with respect to this metric. With this important ingredient, familiar tools from harmonic analysis such as good-$\lambda$ inequalities can be used to prove weighted estimates.

A natural question is whether B\'{e}koll\`{e}'s result can be generalized in a suitable sense to more general classes of domains. Let $\Omega \subset \mathbb{C}^n$ be a pseudoconvex domain with $C^\infty$ defining function $\rho$. In certain situations (see \cite{Mc2}), we can introduce a quasi-distance $d$ on $\Omega$ so that, with respect to the quasi-metric $d$, the Bergman kernel $K(z,w)$ satisfies certain appropriate size and smoothness estimates. In particular, this can be done when the domain is strongly pseudoconvex, convex of finite type, pseudoconvex of finite type in $\mathbb{C}^2$, or decoupled in $\mathbb{C}^n$ . The upshot of this approach is that the theory of Calder\'{o}n-Zygmund operators can be brought to bear when such estimates are proven for the Bergman kernel. McNeal originally used these estimates in \cite{Mc2} to prove $L^p$ bounds on the Bergman projection.  Using the same singular integral theory, these estimates also facilitate the development of an appropriate $\mathcal{B}_p$-type class of weights $\sigma$ for which the Bergman projection $\mathcal{P}$ is bounded on $L^p_\sigma(\Omega)$, which is the focus of this paper. In particular, we prove weighted estimates for a class of domains we call \emph{simple domains} (defined precisely in the following section). The important thing to keep in mind is that in each case we have a quasi-metric $d$ that reflects the geometry of the domain $\Omega$.

By a weight $\sigma$, we mean a locally integrable function on $\Omega$ that is positive almost everywhere. Here we define an appropriate class of weights:

\begin{definition}
For $1<p<\infty$, we say a weight $\sigma$ belongs to the B\'{e}koll\`{e}-Bonami ($\mathcal{B}_p$) class associated to the pseudo-metric $d$ if $\sigma$ and $\sigma'=\sigma^{-1/(p-1)}$ are integrable on $\Omega$ and the following quantity is finite:

$$[\sigma]_{\mathcal{B}_p}:= \sup_{B(w,R); R>d(w,b\Omega)} \left(\frac{1}{\mu(B(w,R))}\int_{B(w,R)} \sigma \mathop{d\mu}\right)\left(\frac{1}{\mu(B(w,R))}\int_{B(w,R)} \sigma^{-1/(p-1)} \mathop{d\mu}\right)^{p-1}.$$

\end{definition}

Our ultimate goal is to prove the following theorem, which is our principal result:

\begin{theorem}\label{main} Let $\Omega$ be a simple domain and $\mathcal{P}$ denote the Bergman projection on $\Omega$. If $\sigma \in \mathcal{B}_p$, then $||\mathcal{P}f||_{L^p_{\sigma}(\Omega)} \lesssim ||f||_{L^p_{\sigma}(\Omega)}$, where $1<p<\infty$.

\end{theorem}

In what follows, since there are so many constants to keep track of we use the notation $A \lesssim B$ to mean that there exists a constant $C$, independent of obvious parameters, so that $A \leq CB$. The symbols $\gtrsim$ and $\approx$ are also used with obvious meanings. 

The authors would also like to acknowledge the work of Chun Gan, Bingyang Hu, and Ilyas Khan, who independently obtained a result, concerning weighted $L^p$ estimates for convex finite type domains, that corresponds to a special case of the main theorem in this paper at around the same time (see \cite{Hu}). Their result for convex domains can be seen to be equivalent to ours, though they use very different machinery and phrase the Muckenhoupt condition on their weights in terms of what they call ``dyadic flow tents." 

\section{Background and Definitions}

All of the domains in this paper are pseudoconvex of finite type in the sense of D'Angelo (see \cite{DA}). In what follows we assume that $\Omega$ is one of the following types of pseudoconvex domains:
\begin{enumerate}
\item strongly pseudoconvex;
\item convex of finite type;
\item finite type in $\mathbb{C}^2$;
\item decoupled in $\mathbb{C}^n.$
\end{enumerate}

Following McNeal in \cite{Mc5}, we will refer to such a domain as a \emph{simple domain}. In \cite{Mc5}, McNeal shows that estimates for the Bergman kernel previously obtained in \cite{Mc3},\cite{Mc1}, and \cite{Mc4} actually fall into a unified framework. It should be noted that historically, estimates for the Bergman kernel on strongly pseudoconvex domains were obtained first, using different methods, for example see \cite{CF}. Strongly pseudoconvex domains were also not one of the types considered in \cite{Mc2}, as the $L^p$ mapping properties of the Bergman projection on these domains were already known (see \cite{PS}). However, in \cite{Mc5} McNeal demonstrates that strongly pseudoconvex domains fall into the same paradigm as the other domains considered. This means that one can use the exact same singular integral machinery as in \cite{Mc2} to prove the $L^p$ regularity of the Bergman projection on strongly pseudoconvex domains, even though this was not originally how this result was obtained. Results on the $L^p$ regularity of the Bergman projection on smooth domains have actually been obtained in a more general context (see \cite{KR}), but in this paper we focus on these \emph{simple domains} since the metric in each of these cases leads to a space of homogeneous type.

We describe, first in qualitative terms, the scaling approach used by McNeal to obtain kernel estimates on all of these domains. Let $U$ be a small neighborhood of a point $p \in b\Omega$ and fix a point $q \in U$. A holomorphic coordinate change $z=\Phi(w)$ with $\Phi(q)=0$ is employed so that $z_1$ is essentially in the complex normal direction (i.e the complex direction normal to $b\Omega$ at $\pi(q)$, where $\pi$ denotes the orthogonal projection to the boundary). In fact, the coordinates can be chosen so $\frac{\partial \rho}{\partial z_1}$ is non-vanishing on $U$. The coordinates $z_2,z_3,\dots,z_n$ are basically the complex tangential directions. The geometric properties of the domain dictate the following: how far can one move in each of the complex directions $z_1,z_2,\dots,z_n$ if one does not want to perturb the defining function $\rho(z)$ by more than $\delta$ (more precisely, a universal constant times $\delta$)? Clearly, one can move no more than some constant multiple of $\delta$ in the radial direction, but it is not at all clear for an arbitrary domain what the answer is for the tangential directions. In fact, roughly speaking, the finite type property of the domain is precisely what ensures that the domain is not ``too flat" and that the amount we can move in the tangential directions is somehow appropriately controlled. We make this notion precise in the following proposition, which can be found in \cite{Mc5}:

\begin{proposition}\label{polydiscs}
Let $\Omega$ be a simple domain. Fix a point $p \in b\Omega$. Then there exists a small neighborhood $U$ such that for sufficiently small $\delta>0$ and any point $q \in U \cap \Omega$, there exist holomorphic coordinates $z=(z_1,z_2,\dots,z_n)$ centered at $q$ and defined on $U$ and quantities $\tau_1(q,\delta),\tau_2(q,\delta),\dots,\tau_n(q,\delta)$ with $\tau_1(q,\delta)=\delta$ so that if we consider the polydisc:
$$P(q,\delta)=\{z \in U: |z_j| < \tau_j(q,\delta), 1 \leq j \leq n\},$$ 

\noindent one has the property that if $z \in P(q,\delta) \cap \Omega$, then $|\rho(z)-\rho(q)| \lesssim \delta$, where the implicit constant is independent of $q$ and $\delta$. Moreover, $\left|\dfrac{\partial \rho}{\partial z_1}\right|> c$ for some $c>0$ on $U \cap \Omega$. In particular, $\dfrac{\partial \rho}{\partial \text{Re}\,z_1 }>c$ on $U \cap \Omega$.
\end{proposition} 

The coordinates $(z_1,z_2,\dots,z_n)$ can depend on $\delta$, for example in the convex finite type case (see \cite{Mc1}), but $z_1$ is always essentially the radial direction. Crucially, the polydiscs also satisfy a kind of doubling property: 

\begin{proposition}[\cite{Mc5},\cite{Mc2}]\label{doubling poly}
There exist independent constants $C,D$ so the following hold for the polydiscs:
\begin{enumerate}
\item If $P(q_1,\delta)\cap P(q_2, \delta) \neq \emptyset$, then $P(q_1,\delta) \subset C P(q_2,\delta)$ and $P(q_2,\delta)\subset C P(q_1,\delta)$.

\item There holds $P(q_1,2 \delta) \subset D P(q_1, \delta).$
\end{enumerate}
\end{proposition}
One can now introduce a local quasi-metric $M$ on $U \cap \Omega$ (see \cite{Mc2}):

\begin{definition} Define the following function on $U \cap \Omega \times U \cap \Omega$:
$$M(z,w)= \inf_{\varepsilon>0}\{\varepsilon: w \in P(z, \varepsilon)\}.$$
Then $M$ defines a quasi-metric on $U \cap \Omega$.
\end{definition}

Note that the volume of a polydisc $P(q,\delta)$ is comparable to $\delta^2 \prod_{j=2}^{n} \left(\tau_j(q,\delta)\right)^2$. In fact this polydisc is comparable to a non-isotropic ball of radius $\delta$ centered at $q$ in the local quasi-metric. To extend this quasi-metric $M$ to a global quasi-metric $d$ defined on $\Omega \times \Omega$, one can just patch the local metrics defined on $U_j \cap \Omega$ together in an appropriate way. The resulting quasi-metric is not continuous, but satisfies all the relevant properties. The balls in this quasi-metric still have volume comparable to a polydisc if they are near the boundary and have small radius. We refer the reader to \cite{Mc2} for more details on this matter. We remark that this metric is technically only defined on $N \times N$, where $N$ is a relative neighborhood of the boundary, in particular the union of the $U_j \cap \Omega$. However, the next lemma shows that this does not present us with any difficulties.

\begin{lemma}\label{reduction to neighborhood}
Let $\mathcal{P} \lvert_{N}$ denote the Bergman projection restricted to $N$; that is, for $f \in L^2(\Omega)$ and $z \in \Omega$,
$$\mathcal{P}(f)(z):= \chi_{N}(z)\int_{N} K(z,w) f(w) \mathop{d \mu(w)},$$ where $K(z,w)$ denotes the Bergman kernel for $\Omega$ and $\chi$ denotes characteristic function.

Then, if $\mathcal{P} \lvert_{N}$ is bounded on  $L^p_{\sigma}(\Omega)$ and $\sigma, \sigma'=\sigma^{-1/(p-1)}$ are integrable on $\Omega$, then $\mathcal{P}$ is bounded on $L^p_{\sigma}(\Omega)$ .
\begin{proof}
Take $f \in L^p_{\sigma}(\Omega)$ and write $f=f_1+f_2$, where $f_1:= f \chi_{N}$ and $f_2:= f \chi_{\Omega \setminus N}$. Then write
\begin{eqnarray*}
||\mathcal{P}f||_{L^p_{\sigma}(\Omega)} & \leq & ||\mathcal{P}f_1||_{L^p_{\sigma}(\Omega)}+ ||\mathcal{P}f_2||_{L^p_{\sigma}(\Omega)} \\
& = & ||\mathcal{P}f_1||_{L^p_{\sigma}(N)}+ ||\mathcal{P}f_2||_{L^p_{\sigma}(N)}+ ||\mathcal{P}f_1||_{L^p_{\sigma}(\Omega \setminus N)}+ ||\mathcal{P}f_2||_{L^p_{\sigma}(\Omega \setminus N)}\\
& = & ||\mathcal{P}\lvert_{N}f||_{L^p_{\sigma}(\Omega)}+ ||\mathcal{P}f_2||_{L^p_{\sigma}(N)}+ ||\mathcal{P}f_1||_{L^p_{\sigma}(\Omega \setminus N)}+ ||\mathcal{P}f_2||_{L^p_{\sigma}(\Omega \setminus N)}\\
& \lesssim & ||f||_{L^p_{\sigma}(\Omega)}+ ||\mathcal{P}f_2||_{L^p_{\sigma}(N)}+ ||\mathcal{P}f_1||_{L^p_{\sigma}(\Omega \setminus N)}+ ||\mathcal{P}f_2||_{L^p_{\sigma}(\Omega \setminus N)}
\end{eqnarray*}
where in the last line we used the hypothesis on $\mathcal{P}\lvert_{N}$. Thus, if we can control the last three terms, we are done. Recall by Kerzman's Theorem, the Bergman kernel extends to a $C^\infty$ function on $\overline{\Omega} \times \overline{\Omega} \setminus \triangle(b \Omega \times b \Omega)$, where $\triangle(b \Omega \times b \Omega)$ denotes the boundary diagonal $\{(z,z): z \in b \Omega \}$(see \cite{Kerz}, \cite{Boas}). Thus, in particular $K(z,w)$ is bounded on compact subsets off the boundary diagonal. We show how this is applied to the term  $||\mathcal{P}f_2||_{L^p_{\sigma}(N)}$, as the other terms can be handled similarly. Then, using this fact about $K(z,w)$, H\"{o}lder's inequality, and the hypotheses on $\sigma$,

\begin{eqnarray*}
||\mathcal{P}f_2||^p_{L^p_{\sigma}(N)}& = & \int_{N} \left| \int_{\Omega \setminus N} K(z,w) f(w) \mathop{d \mu(w)} \right|^p \sigma(z) \mathop{d\mu(z)}\\
& \leq & \int_{N}  \left(\int_{\Omega \setminus N} |K(z,w)| |f(w)| \mathop{d \mu(w)} \right)^p \sigma(z) \mathop{d\mu(z)}\\
& \lesssim & \int_{\Omega}  \left(\int_{\Omega }  |f(w)| \mathop{d \mu(w)} \right)^p \sigma(z) \mathop{d\mu(z)}\\
& = & \sigma(\Omega) \left(\int_{\Omega }  |f(w)| \sigma(w)^{1/p} \sigma(w)^{-1/p} \mathop{d \mu(w)} \right)^p\\
& \leq & \sigma(\Omega) \left(\int_{\Omega} |f(w)|^p \sigma(w) \mathop{d \mu(w)}\right) \left(\int_{\Omega} \sigma'(w) \mathop{d\mu(w)}\right)^{p/q}\\
& = & ||f||_{L^p_{\sigma}(\Omega)}^p \sigma(\Omega) \left(\sigma'(\Omega)\right)^{p/q} \\
& \lesssim & ||f||_{L^p_{\sigma}(\Omega)}^p
\end{eqnarray*}
which establishes the result. 

\end{proof}
\end{lemma} 

This lemma shows that we can reduce to considering $N$ in place of $\Omega$ and $\mathcal{P}\lvert_{N}$ in place of $\mathcal{P}$. Therefore, going forward, we will abuse notation by writing $\Omega$ when we really mean the neighborhood $N$. 

In what follows, let $\mu$ denote Lebesgue area measure on $\Omega$. It is proven in \cite{Mc2} that the triple $(\Omega,d,\mu)$ constitutes a space of homogeneous type. Note that the measure $\mu$ is doubling on the non-isotropic balls essentially because of Proposition \ref{doubling poly}. Note if $d$ is not symmetric, we can symmetrize it by taking $d(z,w)+d(w,z)$ as an equivalent metric. We denote a ball in the quasi-metric $d$ of center $z_0$ and radius $r$ by $$B(z_0,r)=\{z \in \Omega: d(z,z_0)<r\}.$$ Since $\rho$ can be taken to be defined on $\mathbb{C}^n,$ this quasi-metric actually extends to $\overline{\Omega} \times \overline{\Omega}$ because a polydisc can be centered at $q \in b \Omega$ (see, for instance \cite{Mc6} for the convex case). Thus,  for $z \in \Omega$, define $d(z,b\Omega)$ as follows:
$$d(z,b\Omega):= \inf_{w \in b\Omega}d(z,w).$$
It is trivial to verify that for $z, z' \in \Omega$, 
$$d(z,b\Omega) \lesssim d(z',b\Omega)+d(z,z').$$
One can actually show that the distance to the boundary in this quasi-metric is comparable to the Euclidean distance. We have the following lemma.

\begin{lemma}\label{comparability}
Let $\operatorname{dist}(z,b\Omega)$ denote the Euclidean distance of $z$ to the boundary of $\Omega$. Then we have $$d(z,b\Omega) \approx \operatorname{dist}(z,b\Omega).$$

\begin{proof} 
We can assume that $z$ is sufficiently close to the boundary. Let $\pi(z)$ be the normal projection of $z$ to the boundary. Then $d(z,\pi(z)) \lesssim \operatorname{dist}(z,\pi(z))=\operatorname{dist}(z,b\Omega)$ by the structure of the quasi-metric (note that the first coordinate of the polydisc corresponds to the radial direction). This shows the bound $d(z,b\Omega) \lesssim \operatorname{dist}(z,b\Omega)$. 

For the other bound, we only need consider the distance of $z$ to points on the boundary in a local neighborhood $U$ where the local quasi-metric is defined (because otherwise the distances will reduce to Euclidean distance, see \cite{Mc2}). Let $\varepsilon= \operatorname{dist}(z,b\Omega).$ It is clear there is a universal constant $c>0$ so that the shrunken polydisc $P(z,c \varepsilon)$ is strictly contained in $\Omega$. This implies that $d(z,b\Omega) \gtrsim \operatorname{dist}(z,b\Omega)$, as desired.

\end{proof}
\end{lemma}

\noindent The following estimates for the Bergman kernel were obtained by McNeal (see \cite{Mc3,Mc1,Mc4,Mc5} and also \cite{NRSW} for a slightly different approach due to Nagel, Stein, Rosay, and Wainger):

\begin{theorem}\label{kernel estimates}
Let $\Omega$ be a simple domain and $K(z,w)$ denote the Bergman kernel for $\Omega$. Then near any $p \in b\Omega$, there exists a coordinate system centered at $z=(z_1,z_2,\dots,z_n)$ so that if $\alpha$, $\beta$ are multi-indices and $D^\alpha$, $D^\beta$ denote holomorphic derivatives taken in these coordinate directions, we have the following:

$$|D_z^\alpha D_w^\beta K(z,w)| \leq C_{\alpha,\beta} \delta^{-(2+\alpha_1+\beta_1)} \prod_{k=2}^{n} \tau_k(z,\delta)^{-(2+\alpha_k+\beta_k)}$$

\noindent where $\delta= |\rho(z)|+|\rho(w)|+M(z,w)$.

\end{theorem}

 By using the global quasi-metric $d$, one can obtain global estimates on the Bergman kernel. The following was proven in \cite{Mc2}:
 
\begin{theorem}\label{size and smoothness}
Let $\Omega$ be a simple domain. Then the following hold:
\begin{enumerate}
\item \emph{(Size)} There exists a constant $C_1$ so that for all $z,w \in \Omega$:
$$|K(z,w)|\leq \frac{C_1}{\mu(B(z,d(z,w)))}.$$

\item \emph{(Smoothness)} 
There exists a constant $C_2$ and $\nu>0$ so that we have, provided $d(z,w) \geq C_2 d(z,z')$:
$$|K(z,w)-K(z',w)| \leq C_1 \left(\frac{d(z,z')}{d(z,w)}\right)^{\nu} \frac{1}{\mu(B(z,d(z,w)))}.$$

\end{enumerate}
\end{theorem}

We actually get another size estimate for free, which will help us in the course of the proof. This lemma can actually be deduced directly from Theorem \ref{kernel estimates}, but we provide another proof here (which actually shows any domain, not necessarily simple, whose Bergman kernel satisfies the estimates in Theorem \ref{size and smoothness} will necessarily satisfy an additional estimate).

\begin{lemma} Suppose $K(z,w)$ is the Bergman kernel for $\Omega$ and $K$ satisfies the size estimate above. Then there exists a constant $C_3$ so uniformly for all $z, w \in \Omega$  
$$|K(z,w)|\leq C_3 \min \left\{ \frac{1}{\mu(B(z,d(z,b\Omega)))}, \frac{1}{\mu(B(w,d(w,b\Omega)))} \right\}.$$
\begin{proof} Fix $z \in \Omega$. We first claim that given $\varepsilon>0$, there exists a $w' \in \Omega$ so $|K(z,w)|\leq |K(z,w')|$  and $\operatorname{dist}(w',b\Omega)\leq \varepsilon$. The claim follows immediately by applying the  Maximum Principle to the closed domain $\Omega_{\varepsilon}=\{w \in \Omega:|\rho(w)| \geq \varepsilon\}$ and function $K(w,z)=\overline{K(z,w)}$, which is analytic in $w$. 

 Now choose $w' \in b\Omega_\varepsilon$ satisfying the above conditions. Then we have, using Lemma \ref{comparability}: $$d(z,b\Omega) \leq cd(z,w')+cd(w',b\Omega)\leq c'd(z,w')+c'\varepsilon$$
\noindent so we obtain the estimate

$$d(z,w') \geq \frac{1}{c'}d(z,b\Omega)-\varepsilon.$$

Thus, applying the known size estimate, we get 

$$|K(z,w)| \leq |K(z,w')| \leq  \frac{C_2}{\mu(B(z,d(z,w')))}\leq \frac{C_2}{\mu(B(z,\frac{1}{c'}d(z,b\Omega)-\varepsilon))}.$$

Since the inequality above holds for all $\varepsilon>0$ and $\mu$ is doubling on quasi-balls, we obtain

$$|K(z,w)|\leq \frac{C_3}{\mu(B(z,d(z,b\Omega)))}$$

\noindent as desired. Note $C_3$ is independent of $z$. The other inequality follows by symmetry.

\end{proof}
\end{lemma}

\begin{remark} As a clear example of this property, consider the unit ball $\mathbb{B}_n$ where the Bergman kernel is given by $K(z,w)= \frac{1}{(1-\langle z,w\rangle)^{n+1}}$. Then $|K(z,w)| \lesssim \frac{1}{(1-|z|)^{n+1}}$. \end{remark}

It is a well-known fact in harmonic analysis (for example, see \cite{Duong})  that if $B(z,r)$ is a ball of radius $r$, center $z$ in a space of homogeneous type, then there exists uniform constants $c_0$, $m$ so that if $\lambda\geq 1$, we have
$$\mu(B(z,\lambda r)) \leq c_0 \lambda^m \mu(B(z,r)).$$ 

\noindent Here the parameter $m$ can be thought of as roughly corresponding to the ``dimension" of the space. We will use this fact, referred to as the \emph{strong homogeneity property}, in a crucial point in the proof of the main theorem. 

To continue with the analysis, we need to define an appropriate maximal function with respect to the quasi-metric. In analogy with B\'{e}koll\`{e}'s result, we will also only consider balls that touch the boundary of $\Omega$. We make the following definition:
\begin{definition}\label{maximal}
For  $z \in \Omega$ and $f \in L^{1}(\Omega)$, define the following maximal function:
$$\mathcal{M}f(z):=\sup_{B(w,R)\ni z; R>d(w,b\Omega)}\frac{1}{\mu(B(w,R))} \int_{B(w,R)}|f| \mathop{d\mu}.$$
\end{definition}

Proving Theorem \ref{main} can be broken down into the task of proving the following two results (mimicking the approach taken by B\'{e}koll\`{e} in \cite{Bekolle}):
\begin{theorem}\label{maximalth} Let $1<p<\infty$ and suppose $\sigma \in \mathcal{B}_p$. Then $||\mathcal{M}f||_{L^p_\sigma(\Omega)} \lesssim ||f||_{L^p_\sigma(\Omega)}$.
\end{theorem}

\begin{theorem}\label{positive operator} Let $\mathcal{P}^+$ be the positive operator defined $\mathcal{P}^+f(z)=\int_\Omega |K(z,w)| f(w) \mathop{d \mu(w)}$. Let $1<p<\infty$ and $\sigma \in \mathcal{B}_p$. Then $||\mathcal{P}^+f||_{L^p_\sigma(\Omega)} \lesssim ||\mathcal{M}f||_{L^p_\sigma(\Omega)}$.
\end{theorem}

We will prove these two theorems in the following section. It is worth pointing out that Theorem \ref{maximalth} in conjunction with Theorem \ref{positive operator} shows that Theorem \ref{main} actually holds when $\mathcal{P}$ is replaced with $\mathcal{P}^+$, as is typical for Bergman-type operators. 

\section{The Sufficiency of the $\mathcal{B}_p$ Condition}

We begin by proving Theorem \ref{maximalth}. In what follows, we follow the general outline of the approach taken in \cite{Bekolle}. To begin with, we define a regularizing operator $R_k$ for $k \in (0,1)$:

$$R_k(f)(z):= \frac{1}{\mu(B_k(z))}\int_{B_k(z)} |f| \mathop{d \mu}$$ where $B_k(z)=\{w \in \Omega: d(w,z)< k d(z,b\Omega)\}$.

Intuitively, this regularizing operator spreads out the mass of the weight. We will ultimately show it turns $\mathcal{B}_p$ weights into $A_p$ weights. We begin with a simple proposition.

\begin{proposition}\label{triangle1} There exists a constant $C_d >1$ (depending on the quasi-metric $d$) so that if $k \in (0,\frac{1}{2C_d})$, then $z' \in B_k(z)$ implies $z \in B_{k'}(z')$, where $k'=\dfrac{C_d k}{1-C_d k}$.
\begin{proof}
This is a trivial consequence of the triangle inequality. In fact, we can take $C_d=c>1$, where $c$ is the implicit constant in the triangle inequality.
\end{proof}
\end{proposition}
It is also routine to verify that the radius of $B_{k'}(z')$ is at most a fixed multiple of the radius of $B_k(z)$ and the balls have comparable Lebesgue measure, where the implicit constants are independent of $k \in (0,\frac{1}{2C_d}).$
We need another simple proposition to furnish the next lemma. 

\begin{proposition}\label{triangle2} Let $B$ be a ball of radius $r$, center $z_0$, that touches the boundary of $\Omega$ (i.e $r>d(z_0,b\Omega)$). Let $k \in (0,1)$ be fixed. Then there exists an (absolute, independent of $k$) constant $\alpha$ so the dilated ball $\tilde{B}$ with radius $\alpha r$ and center $z_0$ satisfies $\tilde{B}\supset B_k(w)$ for all $w \in B$.
\begin{proof}
Again, the proof is routine. This is also a simple consequence of the triangle inequality.
\end{proof}
\end{proposition}

We are now ready to prove the following significant lemma.

\begin{lemma}\label{inside}
For each $k \in (0,\frac{1}{2C_d})$, we have $\mathcal{M}f(z_0)\lesssim \mathcal{M}(R_{k}(f(z_0)))$ for $z_0 \in \Omega$. The implicit constant is independent of $k$.

\begin{proof}
Fix $k$ and let $B$ be an arbitrary ball touching $b\Omega$ and centered at $z_0$, $\tilde{B}$ an inflation of $B$ with radius chosen as in the previous proposition so that  $\tilde{B} \supset B_{k''}(w)$ for all $w \in B$, where $k''=\frac{k}{c(k+1)}$. Then we have the following:
\begin{eqnarray*}
\mathcal{M}(R_k(f(z_0))) & \geq & \frac{1}{\mu(\tilde{B})}\int_{\tilde{B}} \frac{1}{\mu(B_k(z))}\int_{B_k(z)} |f(w)| \mathop{d \mu(w)} \mathop{d \mu(z)}\\
& = & \frac{1}{\mu(\tilde{B})} \int_{\Omega} \int_{\tilde{B}} \frac{1}{\mu(B_k(z))}\chi_{B_k(z)}(w) |f(w)| \mathop{d \mu(z)} \mathop{d \mu(w)}\\
& \gtrsim & \frac{1}{\mu(\tilde{B})} \int_{B} \int_{\tilde{B}} \frac{1}{\mu(B_{k''}(w))}\chi_{B_{k''}(w)}(z) |f(w)| \mathop{d \mu(z)} \mathop{d \mu(w)}\\
& = &  \frac{1}{\mu(\tilde{B})}\int_{B} |f(w)| \mathop{d \mu(w)}\\
& \approx & \frac{1}{\mu(B)}\int_{B} |f(w)| \mathop{d \mu(w)}
\end{eqnarray*}        

\noindent where we used both propositions and the fact that Lebesgue measure $\mu$ is doubling on quasi-balls. Since the following estimate is true for all balls $B$ centered at $z_0$, the conclusion follows.

\end{proof}
\end{lemma}

We have the additional following lemma which is a straightforward application of Proposition \ref{triangle1} (again the implicit constant is independent of $k$):

\begin{lemma}\label{switch}
Let $f,g$ be positive, locally integrable functions. For each $k \in (0,\frac{1}{2C_d})$, we have the inequality:

 $$\int_{\Omega} f R_k(g) \mathop{d \mu} \lesssim \int_\Omega R_{k'}(f)g \mathop{d \mu},$$
 where $k'=\dfrac{C_d k}{1-C_d k}$.
\begin{proof}
We have:
\begin{eqnarray*}
\int_{\Omega} f R_k(g) \mathop{d \mu} & =& \int_{\Omega} f(z) \frac{1}{\mu(B_k(z))}\int_{B_k(z)} g(w) \mathop{d \mu(w)} \mathop{d \mu(z)}\\
& = &  \int_{\Omega} \int_{\Omega}\frac{f(z)}{\mu(B_k(z))} \chi_{B_k(z)}(w) g(w) \mathop{d \mu(z)} \mathop{d \mu(w)}\\
& \lesssim & \int_{\Omega} \int_{\Omega}\frac{f(z)}{\mu(B_{k'}(w))} \chi_{B_{k'}(w)}(z) g(w) \mathop{d \mu(z)} \mathop{d \mu(w)}\\
& = & \int_{\Omega} g(w) \frac{1}{\mu(B_{k'}(w))}\int_{B_{k'}(w)} f(z) \mathop{d \mu(z)} \mathop{d \mu(w)}\\
& = &  \int_\Omega R_{k'}(f)g \mathop{d \mu},
\end{eqnarray*}
as desired.
\end{proof}
\end{lemma}

The next lemma is fairly straightforward, but does require some care. 

\begin{lemma}\label{outside} Fix $k \in (0,\frac{1}{2C_d})$. Then for any positive, locally integrable function $g$  there holds
$$R_k(\mathcal{M}(g))(z) \approx \mathcal{M}(g)(z), $$
\noindent where the implicit constant is independent of $k$.
\begin{proof}
It suffices to prove that for any fixed $z \in \Omega$, there holds for $w \in B_k(z)$ 

$$\mathcal{M}(g)(w) \lesssim \inf_{z' \in B_k(z)} \mathcal{M}(g)(z') \leq \mathcal{M}(g)(w) ,$$

\noindent where the implicit constant is absolute. Assuming the claim, then
\begin{eqnarray*}
R_k(\mathcal{M}(g))(z) & = & \frac{1}{\mu(B_k(z))}\int_{B_k(z)} \mathcal{M}(g)(w) \mathop{d \mu(w)}\\
& \approx & \frac{1}{\mu(B_k(z))}\int_{B_k(z)} \inf_{z' \in B_k(z)} \mathcal{M}(g)(z')\mathop{d \mu(w)}\\
& = &  \inf_{z' \in B_k(z)} \mathcal{M}(g)(z')\\
& \approx & \mathcal{M}(g)(z).\\
\end{eqnarray*}

Now we prove the claim. The upper bound is trivial. Fix $z \in \Omega$. It is clearly sufficient to show that for any ball $B$ centered at $w \in B_k(z)$ touching the boundary with radius $r$, given any $z'\in B_k(z)$, there is a ball $\tilde{B}$ centered at $z'$ with radius $Cr$ so $\tilde{B} \supset{B}$. First, note that if $B$ touches $b\Omega$ we must have $r \geq \frac{1}{4c} d(z,b\Omega)$, otherwise
$$d(z,b\Omega) \leq c d(z,w)+ cd(w,b\Omega) \leq ck     [d(z,b\Omega)]+cr \leq \frac{1}{2}d(z,b\Omega)+\frac{1}{4}d(z,b\Omega)<d(z,b\Omega),$$
which is absurd. Therefore we may conclude $d(z,b\Omega) \leq 4cr$. 

Thus, if $w' \in B$, we have 

\begin{eqnarray*}
d(w',z') & \leq & c^2[d(w',w)+d(w,z)+d(z,z')]\\
& \leq & c^2[r+2kd(z,b\Omega)]\\
& \leq & 9c^3r
\end{eqnarray*}
so the claim is established by taking $C=9c^3$.

\end{proof} 
\end{lemma}

We will need the following proposition concerning a kind of doubling property for $\mathcal{B}_p$ weights, which appears to be well-known insofar as it is used implicitly in B\'{e}koll\`{e}'s original paper. The proof is largely the same as the proof for the doubling of $A_p$ weights, so we omit it. 

\begin{proposition}\label{doubling}
Suppose $\sigma \in \mathcal{B}_p$. Let $B$ be a pseudo-ball (not necessarily touching $b\Omega$) such that $\lambda B$ touches $b\Omega$, where $\lambda>1$. Then for any $\lambda'>1$, we have

$$\sigma(\lambda'B) \lesssim \sigma(B),$$

\noindent where the implicit constant depends only on $\max\{\lambda,\lambda'\}$. 
\end{proposition}

We now proceed to the proof of Theorem \ref{maximalth}.

\begin{proof}[Proof of Theorem~\ref{maximalth}]
Using the results previously proven, we can make the following progress to proving the theorem, fixing $k \in (0,\frac{1}{2C_d})$ (some of the following implicit constants can depend on $k$, but $k$ is fixed):

\begin{eqnarray*}
\int_{\Omega} [\mathcal{M}(f)(z)]^p \sigma \mathop{d\mu} & \lesssim & \int_{\Omega} [R_k(\mathcal{M}(R_k(|f|)))]^p \sigma \mathop{d \mu}\\
& \leq &  \int_{\Omega} R_k[[\mathcal{M}(R_k(|f|))]^p] \sigma \mathop{d \mu}\\
& \lesssim & \int_{\Omega} [\mathcal{M}(R_k(|f|))]^p R_{k'}(\sigma) \mathop{d \mu}\\
& \lesssim & \int_{\Omega} [\mathcal{M}(R_k(|f|))]^p R_k(\sigma) \mathop{d \mu},
\end{eqnarray*}

\noindent where in the first inequality we use Lemmas \ref{inside} and \ref{outside}, the second inequality is H\"{o}lder, the penultimate inequality is Lemma \ref{switch}, and the last inequality is given by the doubling property of $\sigma$ given in Proposition \ref{doubling}.

Now, if we can prove that the weight $R_k(\sigma)$ belongs to $A_p$, by ordinary weighted theory the last quantity will be controlled by 
$$C_p \int_{\Omega} [R_k(|f|)]^p R_k(\sigma) \mathop{d \mu}.$$
 Assuming this, then we have
\begin{eqnarray*}
\int_{\Omega} [R_k(|f|)]^p R_k(\sigma) \mathop{d \mu} & \leq &\int_{\Omega} R_k(|f|^p \sigma)[R_k(\sigma^{-1/(p-1)})]^{p-1} R_k(\sigma) \mathop{d \mu}\\
& \lesssim & [\sigma]_{\mathcal{B}_p} \int_{\Omega} R_k(|f|^p \sigma) \mathop{d \mu}\\
& \lesssim & [\sigma]_{\mathcal{B}_p} \int_{\Omega} |f|^p \sigma \mathop{d \mu}
\end{eqnarray*}

\noindent where in the first inequality we use H\"{o}lder, the second inequality comes from the fact that [$R_k(\sigma^{-1/(p-1)})]^{p-1} R_k(\sigma) \lesssim [\sigma]_{\mathcal{B}_p}$ (to see this, inflate the balls $B_k(z)$ by at most a fixed amount so they touch the boundary), and for the last step use Lemma \ref{switch}.

Thus, it remains to prove that $R_k(\sigma) \in A_p$. To see this we need to consider two cases for the ball $B(z_0,r)$ over which we take averages: the case where $d(z_0,b\Omega)< 2cr$ (we can inflate the ball so it touches the boundary), and the case where $d(z_0,b\Omega)\geq 2cr$.  For the first case, we proceed as follows:
$$
\frac{1}{\mu(B)}\int_{B} R_k(\sigma) \mathop{d \mu} \lesssim \frac{1}{\mu(B)}\int_{B} \sigma \mathop{d \mu } $$

\noindent using Lemma \ref{switch}, while the other factor is controlled as follows:

\begin{eqnarray*}
\left(\frac{1}{\mu(B)}\int_{B} [R_k(\sigma)]^{-1/(p-1)} \mathop{d \mu}\right)^{p-1} &  = & \left(\frac{1}{\mu(B)}\int_{B} \left(\frac{1}{B_k(z)} \int_{B_k(z)} \sigma(w) \mathop{d \mu(w)}\right)^{-1/(p-1)} \mathop{d \mu(z)}\right)^{p-1}\\
& \leq & \left(\frac{1}{\mu(B)}\int_{B} \frac{1}{B_k(z)} \int_{B_k(z)} \sigma(w)^{-1/(p-1)} \mathop{d \mu(w)}\ \mathop{d \mu(z)}\right)^{p-1}\\
& \lesssim & \left(\frac{1}{\mu(B)}\int_{B} \sigma^{-1/(p-1)} \mathop{d \mu}\right)^{p-1}
\end{eqnarray*}
where for the first inequality we used H\"{o}lder's inequality with negative exponents and the second inequality we used Lemma \ref{switch}. Thus, we clearly have:

$$\left(\frac{1}{\mu(B)}\int_{B} R_k(\sigma) \mathop{d \mu} \right) \left(\frac{1}{\mu(B)}\int_{B} [R_k(\sigma)]^{-1/(p-1)} \mathop{d \mu}\right)^{p-1}\lesssim [\sigma]_{\mathcal{B}_p},$$

\noindent inflating the balls by a fixed amount so they touch the boundary if necessary. 

For the other case, observe $d(z_0,b\Omega) \geq 2cr$, so $r\leq \frac{1}{2c}d(z_0,b\Omega)$. One can verify that given $w \in B$, the balls $B_k(z_0)$ and $ B_k(w)$ have comparable radii. From this it is simple to deduce that if $C_B=R_k(\sigma)(z_0)$, then the following bounds hold for $z \in B$:
$$C_B \lesssim R_k(\sigma)(z) \lesssim C_B$$
where the implicit constants are absolute. It easily follows that $R_k(\sigma) \in A_p$.

\end{proof}

We now state a couple of technical lemmas that will assist us in the proof of Theorem \ref{positive operator}. In particular, they mitigate some difficulties that occur when passing from the proof for the unit ball to the more general cases we consider. 

\begin{lemma}\label{radius control}
 Fix constants $\gamma, \alpha_1, \alpha_2.$ Let $B_0=B(z_0,R_0)$ be a quasi-ball with the property that if $z \in B_0$, then $d(z,b\Omega) \leq \alpha_1 R_0$ and $B_0 \subset B(z, \alpha_1 R_0).$ Define
$$F= \{z \in B_0: \mu(B(z,d(z,b\Omega))) \leq \alpha_2 \gamma \mu(B_0)\}.$$ Then $F \subset \tilde{F}$, where
$$\tilde{F}= \{z \in B_0: d(z,b\Omega) \leq \alpha' \gamma^{\frac{1}{m}} R_0\}$$
and $\alpha'= \alpha_1(c_0 \alpha_2)^{\frac{1}{m}}$. Here $c_0$ and $m$ are the constants in the strong homogeneity property. 
\begin{proof}
Using the strong homogeneity property,

$$\mu(B_0) \leq \mu(B(z,\alpha_1 R_0)) \leq c_0 \left(\frac{\alpha_1 R_0}{d(z,b\Omega)}\right)^m \mu(B(z,d(z,b\Omega))).$$

\noindent If we assume $z \in F$, by the definition of the set $F$, we get an upper bound on $\mu(B(z,d(z,b\Omega)))$, and arrive at the inequality

$$\mu(B_0) \leq c_0 \left(\frac{\alpha_1 R_0}{d(z,b\Omega)}\right)^m \alpha_2 \gamma \mu(B_0).$$

\noindent Thus to avoid absurdity we clearly need 

$$c_0 \left(\frac{\alpha_1 R_0}{d(z,b\Omega)}\right)^m \alpha_2 \gamma \geq 1.$$

\noindent Rearranging this expression, we obtain $$d(z,b\Omega) \leq \alpha' \gamma^{\frac{1}{m}}R_0,$$

\noindent where $\alpha'= \alpha_1(c_0 \alpha_2)^{\frac{1}{m}}$, as required.

\end{proof}

\end{lemma}

\begin{lemma}\label{measure of top of ball}
Let $\alpha'$ be a fixed constant, $\gamma>0$ a constant to be chosen later. Let $B_0=B(z_0,R_0)$ be a pseudo-ball that touches the boundary and $F=\{z \in B_0: d(z, b \Omega)\leq \alpha' \gamma^{1/m} R_0\}$. Then, if $\gamma$ is sufficiently small,
$$\mu(F) \lesssim \gamma^{\frac{1}{m}} \mu(B_0).$$

\begin{proof}
We need to consider two cases: when $R_0$ is large and when $R_0$ is small. We first consider the case when $R_0<R_{\Omega}$ is small, where $R_{\Omega}$ is some appropriately chosen absolute constant that depends only on $\Omega$. We may assume that $B_0$ lies completely in one of the neighborhoods $U$ where the local quasi-metric was constructed. To obtain a favorable estimate on the measure of $F$ in this case, it is easiest to consider the local coordinates constructed by McNeal.

Recall the metric $d$ is constructed by patching together these local metrics, so it suffices to work with the local coordinates on a local level. Recalling $z_0$ denotes the center of $B_0$, we work with coordinates $z=(z_1,z_2,\dots, z_n)$ centered at $z_0$ and with parameter $\delta=R_0$. Note that $B$ can be taken to be $P(z_0,R_0)$, or at least some multiple that will not affect the argument. Let $z_j=x_j+\ii y_j$, $1 \leq j \leq n$. For $z \in P(z_0,R_0)$, write $z=(x_1,y_1,x_2,y_2,\dots,x_n,y_n) \in \mathbb{R}^{2n}$ and write $z'=(y_1,x_2,y_2,\dots,x_n,y_n)$. For $z \in P(z_0,R_0),$ we define the function $R(z')=\sup\{x_1:(x_1,z') \in P(z_0,R_0)\}$. We need to do this because the polydisc may ``extend" past the domain, but we are only considering the measure of the portion that lies in $\Omega$. One can show using geometric arguments that if $z \in F$ then one has the bounds $ R(z')-c'\alpha' \gamma^{\frac{1}{m}}\leq x_1 \leq R(z')$, where $c'$ is some absolute constant. The upper bound is clear by definition. The lower bound follows from the fact that $\frac{\partial \rho}{\partial x_1}>0$ on $U$. Denote by $\sigma_1(z,b \Omega)$ (not to be confused with $\tau_1(z_0,\delta)$) the distance from a point $z$ to $b\Omega$ along the (real) line in the direction of (positive) $x_1$. We show $d(z,b\Omega) \gtrsim \sigma_1(z,b\Omega)$ for all $z$ in this neighborhood. Note if we fix $z \in P(z_0,R_0)$, freezing all the variables except $x_1$, we can select $w=(x_1',y_1,x_2,y_2,\dots,x_n,y_n)$ by increasing $x_1$ so $w \in b\Omega$. Then by the mean value theorem (in one real variable), there is a point $\zeta$ in the neighborhood $U$ so  

\begin{eqnarray*}
d(z,b\Omega) & \approx & |\rho(z)|\\
& = & |\rho(z)-\rho(w)|\\
& = &  \left|\frac{\partial \rho(\zeta)}{\partial x_1}\right| |x_1'-x_1|\\
& \approx & x_1'-x_1\\
& = & \sigma_1(z,b\Omega)
\end{eqnarray*}

\noindent where the implicit constant is independent of $z$.
Crucially we use the fact that $ \frac{\partial \rho}{\partial x_1}$ is bounded away from zero by the coordinate construction. This shows that $\sigma_1(z,b\Omega) \lesssim d(z,b\Omega)$ and establishes the claim.

Thus, we can gain control on the measure of $F$ by integrating in these coordinates, using Fubini and noting the function $R(z')$ will vanish after the first variable is integrated:
\begin{eqnarray*} \mu(F) & \lesssim &  \int_{|z_n| \leq \tau_n(z_0,R_0)} \int_{|z_{n-1}| \leq \tau_{n-1}(z_0,R_0)} \dots \int_{|y_1| \leq R_0} \int_{R(z')-c'\alpha'\gamma^{1/m}R_0 \leq |x_1| \leq  R(z')}
\mathop{d x_1} \mathop{d y_1} \dots \mathop{d y_{n-1}} \mathop{d y_n}
\\
&\approx &  \gamma^{\frac{1}{m}} R_0^2 \prod_{j=2}^{n} \left(\tau_j(z,R_0)\right)^2 \\
& \approx & \gamma^{\frac{1}{m}} \mu(B_0)
\end{eqnarray*}
which yields the required estimate.

Now suppose that $R_0 \geq R_{\Omega}$. Since we are assuming $\gamma$ is small, we can cover $F$ and $b\Omega$ with finitely many small (Euclidean balls) so that in each ball, the normal projection to the boundary is well-defined. Then, in each of these balls with center $z_c$ we can introduce a smooth change of coordinates $z=(z_1,\dots,z_n)$ centered at $\pi(z_c)$, where $\pi$ denotes the normal projection to the boundary, so we have $x_1$ is in the real normal direction at $\pi(z_c)$ and the coordinates $z_2,\dots z_n$ lie in the real tangent plane at $\pi(z_c)$. A similar type of coordinate system is employed in \cite[Lemma 4.1]{LS}. In each ball, we can perform an integration very similar to the one above in these coordinates and obtain $ \mu(F) \leq C_\Omega \gamma^\frac{1}{m} R_0$, where $C_\Omega$ is a constant depending only on the ambient domain. Since $R_0$ is uniformly bounded above and below by assumption, we also have $\mu(B_0)$ is bounded above and below by a universal constant for the domain. Thus, we can deduce that $\mu(F) \lesssim \gamma^{\frac{1}{m}}\mu(B_0)$, as desired.
\end{proof}

\end{lemma} 

Next, we proceed to prove Theorem \ref{positive operator}. In what follows, we consider the positive Bergman operator

$$\mathcal{P}^+f(z)=\int_\Omega |K(z,w)| f(w) \mathop{d \mu(w)}.$$

It is known for the strongly pseudoconvex and convex finite type cases that the positive operator $\mathcal{P}^+$ is bounded on $L^p(\Omega)$, $1<p<\infty$ (see \cite{McSt}, \cite{PS}). We remark that our proof obtains the same result for the other cases in addition to the weighted estimates (just take $\sigma=1$).

\begin{proof}[Proof of Theorem \ref{positive operator}]
We proceed by proving a good-$\lambda$ inequality as in classical singular integral theory and B\'{e}koll\`{e}'s paper. In particular, we will show that there exist positive constants $C$ and $\delta$ so that given any $f \in L^1(\Omega)$ and $\lambda, \gamma>0$ we have

$$\sigma\left( \{ \mathcal{P}^+f>2 \lambda \text{ and } \mathcal{M}f\leq  \gamma \lambda \}\right) \leq C \gamma^\delta \sigma\left(\{\mathcal{P}^+f>\lambda\}\right).$$

By the regularity of $\sigma$, it suffices to prove 

$$\sigma\left( \{z \in \mathcal{O}: \mathcal{P}^+f>2 \lambda \text{ and } \mathcal{M}f\leq  \gamma \lambda \}\right) \leq C \gamma^\delta \sigma\left(\mathcal{O} \right).$$

\noindent for any open set $\mathcal{O}$ containing $\{z \in \Omega: \mathcal{P}^+f>\lambda\}$. Applying a Whitney decomposition to $\mathcal{O}$, consider a fixed ball $B_0$ in the Whitney decomposition with center $z_0$ and radius $R_0$. It suffices to show 

$$\sigma\left( \{z \in B_0: \mathcal{P}^+f>2 \lambda \text{ and } \mathcal{M}f\leq  \gamma \lambda \}\right) \leq C \gamma^\delta \sigma\left(B_0 \right).$$

We may assume that there exists a $\zeta_0 \in B_0$ so that $\mathcal{M}f(\zeta_0)\leq \gamma \lambda$, otherwise the inequality is trivial. Also note we are free to take $\gamma$ sufficiently small as the inequality is trivial for large $\gamma$. By properties of the Whitney decomposition, we know that for some inflation constant $c_1>1$, the ball $\tilde{B_0}$ with radius $c_1R_0$ contains a point $z'$ so that $\mathcal{P}^+f(z') \leq \lambda$. Finally, let $c_2$ be chosen large enough so that the ball centered at $z'$ with radius $c_2R_0$ contains $B_0$ and let $\overline{B}_0$ be the ball centered at $ z'$ with radius equal to $\rho= \max\{ d(z',b\Omega), c_2R_0 \}$. Without loss of generality we may assume $c_2>>c_1>>c$.

Write $f=f_1+f_2$ where $f_1=f\chi_{\overline{B}_0}$ and $f_2= f\chi_{\Omega \setminus \overline{B}_0}$. Without loss of generality, we may assume $f$ is positive. We first show there exists an absolute constant $A$ so that for $z \in B_0$, $\mathcal{P}^+f_2(z) \leq \lambda+A\gamma \lambda.$

We have, for $z \in B_0$,
\begin{eqnarray*}
\mathcal{P}^+f_2(z)\ & = & \int_{\Omega \setminus \overline{B}_0} |K(z,w)| f(w) \mathop{d \mu(w)}\\
& \leq & \int_{\Omega} |K(z',w)| f(w) \mathop{d \mu(w)}+\int_{\Omega \setminus \overline{B}_0} |K(z,w)-K(z',w)| |f(w)| \mathop{d \mu(w)}.
 \end{eqnarray*}
 
Obviously, for the first term we have

$$\int_{\Omega} |K(z',w)| f(w) \mathop{d \mu(w)}=\mathcal{P}^+f(z')<\lambda.$$

The second term is handled as follows. First notice that if $w \in \Omega \setminus \overline{B}_0$, we have $d(z,w) \geq C_2 d(z,z')$, provided $c_2$ is taken appropriately large. Also, it can be shown $d(z,w) \gtrsim \rho$. For $0\leq k<\infty$, let $$A_k=\{w \in \Omega: 2^k \rho'  \leq d(z,w) \leq 2^{k+1} \rho' \} $$ 
\noindent where $\rho'=\inf_{w \in \Omega \setminus \overline{B}_0}d(z,w) \approx \max\{C_2 d(z,z'),\rho\}$. Then we estimate:

\begin{eqnarray*} \int_{\Omega \setminus \overline{B}_0} |K(z,w)-K(z',w)| |f(w)| \mathop{d \mu(w)} & \leq  & \int_{\Omega \setminus \overline{B}_0}\left(\frac{d(z,z')}{d(z,w)}\right)^{\nu}\frac{|f(w)|}{\mu(B(z,d(z,w)))} \mathop{d \mu(w)}\\
& \leq & \sum_{k=0}^{\infty} \int_{A_k}\left(\frac{d(z,z')}{d(z,w)}\right)^{\nu}\frac{|f(w)|}{\mu(B(z,d(z,w)))} \mathop{d \mu(w)}\\
& \lesssim & \sum_{k=0}^{\infty} \int_{A_k}2^{-k \nu}\frac{|f(w)|}{\mu(B(z,2^k \rho'))} \mathop{d \mu(w)}\\
& = & \sum_{k=0}^{\infty} \frac{2^{-k \nu}}{\mu(B(z,2^{k+1} \rho' ))} \int_{A_k}\frac{|f(w)|\mu(B(z,2^{k+1}\rho'))}{\mu(B(z,2^k \rho'))} \mathop{d \mu(w)}\\
& \lesssim & \mathcal{M}f(\zeta_0)\\
& \leq & \gamma \lambda.
\end{eqnarray*}

Now we must consider some cases. First consider the case when $d(z', b\Omega) \geq c_2 R_0$. We then have the easy estimate:
\begin{eqnarray*}
\mathcal{P}^+f_1(z) & = & \int_{\overline{B}_0} |K(z,w)||f(w)| \mathop{d \mu(w)}\\
& \leq & \frac{1}{\mu(B(z,d(z,b\Omega)))}\int_{\overline{B}_0} |f(w)| \mathop{d \mu(w)}\\
& \lesssim &  \frac{1}{\mu(\overline{B}_0)}
}\int_{\overline{B}_0} |f(w)| \mathop{d \mu(w)\\
& \lesssim & \mathcal{M}f(\zeta_0)\\
& \leq & \gamma \lambda.
\end{eqnarray*}

By choosing $\gamma$ sufficiently small, it is clear we can make the left hand side of the good-$\lambda$ inequality equal to $0$, so the inequality is trivial in this case.

Now for the other case suppose that $d(z', b\Omega) < c_2 R_0.$ Note that if $\mathcal{P}^+f(z)>2 \lambda$, then by what we have shown above $\mathcal{P}^+f_1(z)>b \lambda$ where $b=2-(1+A \gamma)$. We estimate:
\begin{eqnarray*}
b \lambda & < & \mathcal{P}^+f_1(z)\\
& \leq & \int_{\overline{B}_0} |K(z,w)||f(w)| \mathop{d \mu(w)}\\
& \leq & \frac{1}{\mu(B(z,d(z,b\Omega)))} \int_{\overline{B}_0} |f(w)| \mathop{d \mu(w)}\\ 
& = & \frac{\mu(B(z',c_2R_0))}{\mu(B(z,d(z,b\Omega)))} \frac{1}{\mu(\overline{B}_0)}\int_{\overline{B}_0} |f(w)| \mathop{d \mu(w)}\\
& \lesssim & \frac{\mu(B(z',c_2R_0))}{\mu(B(z,d(z,b\Omega)))} \mathcal{M}f(\zeta_0)\\
& \leq & \frac{\mu(B(z',c_2R_0))}{\mu(B(z,d(z,b\Omega)))} \gamma \lambda.
\end{eqnarray*}

This implies the following: 

$$\mu(B(z,d(z, b\Omega)))\lesssim \gamma \mu(B(z',c_2R_0))\lesssim \gamma \mu(B(z,c_2R_0)).$$

Let $$F= \{z \in B_0: \mu(B(z,d(z,b\Omega))) \leq \alpha \gamma \mu(B(z,c_2R_0))\}$$
where $\alpha$ is the implicit constant above. By renaming $\alpha$, we can replace $\mu(B(z,c_2R_0))$ by $\mu(B_0)$, using the doubling property. Note that by the above we have proven

$$\{z \in B_0: \mathcal{P}^+f(z)> 2 \lambda \text{ and } \mathcal{M}f(z)\leq \gamma \lambda\} \subset F.$$

We need to prove that we have good control over the measure of the set $F$. In particular, we claim $\mu(F) \lesssim \gamma^{\frac{1}{m}} \mu(B_0)$ where we recall $m$ is the exponent, characteristic of the domain, that appears in the polynomial growth condition in the measure $\mu$. By Lemma \ref{radius control}, we can replace $F$ with $\tilde{F}= \{z \in B_0: d(z,b\Omega) \leq \alpha' \gamma^{\frac{1}{m}} R_0\}$.  By inflating $B_0$ if necessary, we can assume without loss of generality $R_0>d(z_0,b\Omega)$ so that $B_0$ touches $b\Omega$. Then Lemma \ref{measure of top of ball} establishes the claim. 

Now we prove that $\mathcal{B}_p$ weights satisfy a kind of ``fairness" property that is characteristic of $A_\infty$ weights.  As in the previous proofs, define a regularized weight as follows:
$$\sigma'(z)=R(\sigma)(z)= \frac{1}{\mu(B(z))} \int_{B(z)} \sigma(w) \mathop{d \mu(w)},$$

\noindent where $B(z)=\{w \in \Omega: d(w,z)< k_0 d(z,b\Omega)\}$ for some appropriately chosen constant $k_0$.

Recall that by previous work, $\sigma' \in A_p$. First we show $\sigma'(B_0) \lesssim \sigma(B_0).$ Using basically the arguments of Lemma \ref{switch}, we can show that $$\sigma'(B_0) \lesssim \int_{\Omega} \sigma(\zeta) \frac{1}{\mu(B'(\zeta))}\int_{B_0 \cap B'(\zeta)} \mathop{d \mu(z)} \mathop{d \mu(\zeta)},$$

\noindent where $B'(\zeta)$ is some fixed inflation of $B(\zeta)$. We claim that we can inflate $B_0$ by a fixed amount to a ball $\hat{B}_0$ so that $\zeta \notin \hat{B}_0$  implies $B'(\zeta) \cap B_0= \emptyset$.

Then $$\sigma'(B_0) \lesssim \sigma(\hat{B_0}) \lesssim \sigma(B_0)$$ using the doubling property of $\sigma$.  We can use a similar argument to verify that $\sigma(F) \lesssim \sigma'(F)$. In particular, one can check that 

$$\sigma(F) \lesssim \int_{\Omega} \frac{1}{\mu(B(\zeta))} \int_{B(\zeta) \cap F} \sigma(z) \mathop{d \mu(z)} \mathop{d \mu(\zeta)}.$$

One can check there exists a constant $K_1 $ and an inflated ball $B_0'$ so that if we define the set $\hat{F}$

$$\hat{F}=\{z \in B_0': d(z,b\Omega) \leq K_1 \gamma^{\frac{1}{m}} R_0\}$$

\noindent then $\hat{F}$ has the property that if $\zeta \notin \hat{F}$, then $B(\zeta) \cap F= \emptyset$. Then $\sigma(F) \lesssim \sigma'(\hat{F}).$ Note that by the reasoning leading to the computation of the Lebesgue measure of $F$, $ \mu(\hat{F}) \lesssim \gamma^{\frac{1}{m}} \mu(B).$ Then notice we obtain, by the fairness property of $A_p$ weights 

$$\sigma(F) \lesssim \sigma'(\hat{F}) \lesssim  [\mu(\hat{F})/\mu(B_0')]^\delta \sigma'(B_0') \lesssim [\mu(\hat{F})/\mu(B_0)]^\delta \sigma(B_0),$$

\noindent and $[\mu(\hat{F})/\mu(B_0)]\lesssim \gamma^{\frac{1}{m}}$. Thus, the good-$\lambda$ inequality is demonstrated, renaming $\delta$ as $\frac{\delta}{m}$. The rest of the proof follows from standard relative distribution estimates.

\end{proof}

\begin{remark}
In principle one could track constants in the proof of sufficiency and obtain an upper quantitative estimate for the norm of $\mathcal{P}$ or $\mathcal{P}^+$ on $L^p_\sigma(\Omega)$ in terms of $[\sigma]_{\mathcal{B}_p}$. However, such an estimate would almost certainly not be sharp. We resolve this issue in \cite{HWW} using modern techniques of dyadic harmonic analysis as in \cite{RTW}.
\end{remark}

\section{The Necessity of the $\mathcal{B}_p$ Condition}

We would now like to consider whether the condition $\sigma \in \mathcal{B}_p$ is necessary for $\mathcal{P}$ to be bounded on $L^p_{\sigma}(\Omega)$. In what follows we obtain a partial answer to this question, valid for any simple domain $\Omega$. In the special case that $\Omega$ is strongly pseudoconvex, we will prove that the $\mathcal{B}_p$ condition is necessary. In general, we require additional hypotheses, in particular a lower bound on the kernel and the integrability of $\sigma$ and its dual, for our proof technique. We first prove a lemma which is valid for any simple domain where the Bergman kernel satisfies an appropriate lower estimate. This lemma is an analogue of \cite[Lemma 5]{Bekolle} and essentially the same argument is given.

\begin{lemma} \label{small balls} Suppose the Bergman kernel $K(z,w)$ on a simple domain $\Omega$ satisfies the following property: if $$\max\{d(z,b\Omega), d(w,b\Omega)\} \lesssim d(z,w)$$ and $d(z,w)$ is small enough,
\noindent then we have
$$|K(z,w)| \gtrsim \frac{1}{\mu(B(w,d(z,w)))},$$
where the implicit constants are universal for $\Omega$.
Let $B_1(\zeta_0,R)$ be a ball of small radius $R < \varepsilon_0$ touching $b\Omega$. Then there exists a ball $B_2$ of the same radius, touching $b\Omega$ with $d(B_1,B_2) \approx R$ so that if $f \geq 0$ is a function supported in $B_i$ and $z \in B_j$, with $i \neq j$ and $i,j \in \{1,2\}$, then we have

$$|\mathcal{P}[f](z)| \gtrsim \frac{1}{\mu(B_i)} \int_{B_i} f(w) \mathop{d \mu(w)}.$$

\begin{proof}
For simplicity, suppose $i=1$. Choose $B_2$ so that if $z \in B_2$ and $\zeta \in B_1$, we have the estimate $d(\zeta_0,z) \geq C_2 d(\zeta_0, \zeta)$, where $C_2$ is the constant that appears in the smoothness estimate. Then, estimate as follows (assuming $C_2$ is appropriately large):

\begin{eqnarray*}  |\mathcal{P}[f](z)| & = & \left|\int_{B_1}K(z,\zeta) f(\zeta) \mathop{d \mu(\zeta)}\right| \\
& \geq & \left|\int_{B_1}K(z,\zeta_0) f(\zeta) \mathop{d \mu(\zeta)}\right|-\int_{B_1}|K(z,\zeta_0)-K(z,\zeta)| f(\zeta) \mathop{d \mu(\zeta)}\\
& \geq & |K(z,\zeta_0)|\int_{B_1} f(\zeta) \mathop{d \mu(\zeta)}-C_1\int_{B_1}\left(\frac{d(\zeta_0,\zeta)}{d(\zeta_0,z)}\right)^{\nu} \frac{1}{\mu(B(\zeta_0,d(\zeta_0,z)))} f(\zeta) \mathop{d \mu(\zeta)}\\
& \geq & |K(z,\zeta_0)|\int_{B_1} f(\zeta) \mathop{d \mu(\zeta)}-\frac{C_1}{C_2^\nu\mu(B(\zeta_0,d(\zeta_0,z)))}\int_{B_1} f(\zeta) \mathop{d \mu(\zeta)}\\
& \gtrsim & \frac{1}{\mu(B(\zeta_0,d(\zeta_0,z)))} \int_{B_1} f(\zeta) \mathop{d \mu(\zeta)}\\
& \approx & \frac{1}{\mu(B_1)} \int_{B_1} f(w) \mathop{d \mu(w)}.
\end{eqnarray*}

\noindent Note in the penultimate estimate we use the hypothesis of the lower bound on the kernel.
\end{proof}

\end{lemma}

Using this lemma we obtain the following theorem, which grants the necessity of the $\mathcal{B}_p$ condition under certain conditions.

\begin{theorem}\label{necessity} Suppose the Bergman kernel $K(z,w)$ on a simple domain $\Omega$ satisfies the lower bound in Lemma \ref{small balls}. Then if $\mathcal{P}$ maps $L^p_\sigma(\Omega)$ to $L^p_\sigma(\Omega)$ and additionally $\sigma$ and $\sigma^{-\frac{1}{p-1}}$ are integrable, we must have $\sigma \in \mathcal{B}_p.$

\begin{proof}
We follow closely a standard argument in harmonic analysis that is used, for example, in proving the necessity of the $A_p$ condition for the Hilbert/Riesz transforms (see, for example, the proof of \cite[Theorem 7.47]{G})). 

First, we note that the assumption that $\sigma$ and its dual are integrable allows us to consider only small balls as in the proof of Lemma \ref{small balls} when we compute the $\mathcal{B}_p$ characteristic. Let $B_1$ and $B_2$ be two small balls as considered in the lemma, and $f$ a positive function supported on $B_1$. For notational convenience, let $\langle f \rangle_{B}$ denote the average of $f$ over $B$. Note that Lemma \ref{small balls} implies:

$$B_2 \subseteq \{\mathcal{P}(f)(z)\geq c \langle f \rangle_{B_1}\}$$

\noindent where $c$ is the implicit constant in the lemma. Let $\mathcal{A}=||\mathcal{P}||_{L^p_\sigma(\Omega)}$. Using the fact that $\mathcal{P}$ is bounded on $L^p_\sigma(\Omega)$, we obtain:

\begin{equation} \sigma(B_2) \lesssim \frac{\mathcal{A}^p }{\left(\langle f \rangle_{B_1}\right)^p}\int_{\Omega} |f|^p \sigma \mathop{d \mu} \tag{1}\label{1}.\end{equation} 

Note we may interchange the roles of $B_1$ and $B_2$ to obtain

\begin{equation}\sigma(B_1) \lesssim \frac{\mathcal{A}^p }{\left(\langle f \rangle_{B_2}\right)^p}\int_{\Omega} |f|^p \sigma \mathop{d \mu}\tag{2}\label{2}.\end{equation}

Now take $f= \chi_{B_2}$ to obtain $\sigma(B_1) \lesssim \mathcal{A}^p \sigma(B_2)$. Then substitute this into \eqref{1} to obtain 

\begin{equation}\sigma(B_1) \lesssim \frac{\mathcal{A}^{2p} }{\left(\langle f \rangle_{B_1}\right)^p}\int_{\Omega} |f|^p \sigma \mathop{d \mu} \tag{3} \label{3}.\end{equation}

Finally, take $f=\sigma^{-\frac{1}{p-1}} \chi_{B_1}$ and substitute into $\eqref{3}$ to obtain

$$ \langle \sigma \rangle_{B_1} \left ( \langle \sigma^{-\frac{1}{p-1}} \rangle_{B_1} \right)^{p-1} \lesssim \mathcal{A}^{2p}$$

\noindent which completes the proof.

\end{proof}

\end{theorem}

We next show that if $\Omega$ is strongly pseudoconvex and $\mathcal{P}$ is bounded on $L^p_\sigma(\Omega)$, then it follows that $\sigma, \sigma^{-1/p-1}$ are integrable on $\Omega$. 

\begin{lemma}\label{integrability}Let $\Omega$ be strongly pseudoconvex with smooth boundary.
Suppose $\mathcal{P}$ is bounded on $L^p_\sigma(\Omega)$. Then $\sigma, \sigma^{-1/p-1} \in L^1(\Omega).$
\begin{proof}

It suffices to prove $\sigma^{-\frac{1}{p-1}} \in L^1(\Omega)$. Then the integrability of $\sigma$ follows by a duality argument. Indeed, if $\mathcal{P}$ is bounded on $L^p_{\sigma}(\Omega)$, then since the Bergman projection is self-adjoint $\mathcal{P}$ is also bounded on $L^q_{\sigma'}(\Omega)$, where $q$ is the dual exponent to $p$ and $\sigma'=\sigma^{-\frac{1}{p-1}}$. The same arguments then imply that $(\sigma^{-\frac{1}{p-1}})^{-\frac{1}{q-1}}=\sigma$ is integrable. 

We first claim that there exists an $\varepsilon>0$ so that for any $w\in \Omega$, there exists a point $z_0 \in \Omega$ (depending on $w$) so that for all $z$ in a small neighborhood of $z_0$ (call it $N_{z_0}$) and $w' \in B(w, \varepsilon)$, we have $|K(z,w')|\approx 1$ and for any $z_1,z_2 \in N_{z_0}$ and $w' \in B(w,\varepsilon)$, $\arg\{ K(z_1,w'),K(z_2,w')\} \in [-\frac{1}{3},\frac{1}{3}].$ Here $B(w,\varepsilon)$ denotes the Euclidean ball of radius $\varepsilon$. To see this, note that if $z_0$ is chosen so $\operatorname{dist}(z_0,b\Omega)> 1,$ then $|K(z,w')|\lesssim 1$ by Kerzman's result that the Bergman kernel extends to a $C^\infty$ function off the boundary diagonal. So it remains to show that there exists an $\varepsilon>0$ so $|K(z,w')|\gtrsim 1$ for $z,w'$ as above, and that the argument condition is satisfied. The argument condition again follows from Kerzman's theorem, perhaps by shrinking $N_{z_0}$ sufficiently small. Suppose the remainder of the claim is not true. Then there is a sequence of points $w_n$ so for each $z$ satisfying $\operatorname{dist}(z,b \Omega)>1$ and $n$, there is a point $w_n' \in B(w_n, \frac{1}{n})$ so that $|K(z,w_n')|< \varepsilon_n$, where $\varepsilon_n$ is a sequence that tends to $0$. Passing to a subsequence, we have that $w_n' \rightarrow w'' \in \bar{\Omega}$ with $K(z, w'')=0$ for all $z$ with $\operatorname{dist}(z,b \Omega)>1$ (note that $w_n'$ depends on $z$ but the limit point $w''$ does not). First consider the case when $w'' \in \Omega$. Then we immediately get a contradiction, since $\{z: \operatorname{dist}(z,b \Omega)>1\}$ is open in $\mathbb{C}^n,$ while the zero set of $K(\cdot,w'')$ is a complex variety of complex codimension one (note $K(\cdot,w'')$  is not identically zero). 

Note that in fact we can repeat this procedure for each $n$ taking $z_0$ so $\operatorname{dist}(z_0,b\Omega)> \frac{1}{n}$. Then in fact we will obtain a sequence of limit points $w''_n$. Then passing to a subsequence if necessary, we can assume that $w''_n \rightarrow w^* \in \bar{\Omega}$. By the argument above, we may assume $w^* \in b\Omega$. For each $n$, we can select a $z_n$ so $\operatorname{dist}(w^*,z_n)\leq \frac{2}{n}$ and $\operatorname{dist}(z_n, b\Omega) > \frac{1}{n}$. Then clearly $z_n \rightarrow w^*$ and also $K(z_n,w''_n)=0$ for all $n$. Looking at the asymptotic expansion for the Bergman kernel in the strongly pseudoconvex case obtained in \cite{Bout}, we see that this is impossible. In particular, the asymptotic expansion takes the following form:
$$K(z,w)=a(z,w)\psi(z,w)^{-n-1},$$ where $a$ is continuous on $\overline{\Omega} \times \overline{\Omega}$ and is non-vanishing on $\triangle(b \Omega \times b \Omega)$, and $\psi$ is $C^\infty$ on $\overline{\Omega} \times \overline{\Omega}$ with certain additional properties. In particular, $\psi$ vanishes on the boundary diagonal. Thus, clearly we must have $a(w^*,w^*)=0$. But this is impossible as $a$ does not vanish on the boundary diagonal. This establishes the claim.

We now show that the claim implies the integrability of $\sigma^{-\frac{1}{p-1}}$. First, let $f \in L^p_\sigma(\Omega)$ be a positive function. We claim $f \in L^1(\Omega)$. Fix $w \in \Omega$ and let $\varepsilon$ and $z_0$ be as in the above claim. Then the function $F(w'):=K(z_0,w')^{-1}f(w')\chi_{B(w,\varepsilon)}(w') \in L^p_\sigma(\Omega)$ by the claim. Notice

$$\mathcal{P}(F)(z)= \int_{\Omega \cap B(w,\varepsilon)} \frac{K(z,w')}{K(z_0,w')} f(w') \mathop{d \mu(w')}$$

\noindent is in $L^p_\sigma(\Omega)$ by hypothesis and hence is finite almost everywhere. Thus in particular there exists a $z'$ in $N_{z_0}$ so $|\mathcal{P}(f)(z')|<\infty$. But then this implies, using the argument condition,

$$\left| \int_{\Omega \cap B(w,\varepsilon)} f(w') \mathop{d \mu(w')}\right|<\infty.$$

It is then possible to choose a finite covering $B(w_1,\varepsilon),\dots B(w_n,\varepsilon)$ of $\Omega$, which thus implies $f \in L^1(\Omega)$. 

Now, suppose to the contrary that $\sigma^{-1/p-1}$ is not integrable. Then there exists a positive function $g \in L^p(\Omega)$ so that $\int_{\Omega} g \sigma^{-1/p} \mathop{d \mu}=\infty$. But then taking $f=g\sigma^{-1/p}$, we see $f \in L^p_\sigma(\Omega)$. This implies $f \in L^1(\Omega)$, a contradiction since we know $f \notin L^1(\Omega)$.
\end{proof}
\end{lemma}

Finally, we show that strongly pseudoconvex domains also satisfy the necessary lower bound on the Bergman kernel, so the $\mathcal{B}_p$ condition is both necessary and sufficient in this case.

\begin{corollary}
 Let $\Omega$ be strongly pseudoconvex with smooth boundary. Then if $\mathcal{P}$ is bounded on $L^p_\sigma(\Omega)$, $\sigma \in \mathcal{B}_p$.
 \begin{proof}
Throughout the proof, we assume that $d(z,w)$ is chosen sufficiently small and that $\max\{ d(z,b\Omega),d(w,b\Omega)\} \lesssim d(z,w)$. As above, by a result of Boutet and Sj\"{o}strand (\cite{Bout}), we have $K(z,w)=a(z,w)\psi(z,w)^{-n-1}$, where $a$ is $C^\infty$ on $\overline{\Omega} \times \overline{\Omega} \setminus \triangle(b \Omega \times b \Omega)$ and continuous on $\overline{\Omega} \times \overline{\Omega}$, $a$ does not vanish on the diagonal sufficiently close to the boundary, and $\psi$ is a $C^\infty$ function with $\psi(z,z)= -\rho(z)$, and the additional condition that $\partial_w \psi$, $\overline{\partial}_z \psi$ are vanishing of infinite order on the diagonal $w=z$.
We claim that if we choose $d(z,w)$ small enough then we have $|\psi(z,w)| \lesssim d(z,w)$. To see this, note that Taylor's theorem together with the conditions on $\psi$ imply
$$|\psi(z,w)| \leq |\rho(w)|+ \left|\sum_{j=1}^n \frac{\partial \rho(w)}{\partial z_j}(z_j-w_j)+\frac{1}{2}\sum_{j,k=1}^{n}\frac{\partial^2 \rho(w)}{\partial z_j \partial z_k}(z_j-w_j)(z_k-w_k)\right|+\mathcal{O}(|z-w|^2).$$
On the other hand, the quasi-metric can be explicitly written down (locally) using a biholomorphic change of coordinates centered at $w$ (see \cite{Mc5}). First, we may by a unitary rotation plus normalization and translation assume $\partial \rho(w)=d z_1$ and $w=0$. Then in these coordinates, $\sum_{j=1}^{n}\frac{\partial \rho (w)}{\partial z_j}(z_j-w_j)=z_1$. Then, define holomorphic coordinates $\zeta= (\zeta_1,\dots,\zeta_n)$ as follows:
$$\zeta_1=z_1+\frac{1}{2}\sum_{j,k=1}^{n}\frac{\partial^2 \rho(w)}{\partial z_j \partial z_k}z_jz_k, \hspace{0.2 cm} \zeta_j=z_j, j=2,\dots,n.  $$In particular, 
$$d(z,w)\approx |z_1-w_1|+\sum_{j=2}^{n}|z_j-w_j|^{2},$$
\noindent where the components of $z$ and $w$ are computed in $\zeta$ coordinates.
Since $|\rho(w)| \lesssim d(z,w)$ by hypothesis, then it is clear, applying the change of variables, that $|\psi(z,w)| \lesssim d(z,w).$ Finally, it is easy to verify that in the strongly pseudoconvex case, $\mu(B(z,r)) \approx r^{n+1}$, because $\tau_j(z,\delta)=\delta^{1/2}$ for $j=2,\dots,n$ (see \cite{Mc5} for instance). Therefore, if $d(z,w)$ is chosen appropriately small we can obtain the following lower bounds:
\begin{eqnarray*}
|K(z,w)| & \gtrsim & |\psi(z,w)|^{-n-1}\\
& \gtrsim & (d(z,w))^{-n-1}\\
& \approx & \frac{1}{\mu(B(w,d(z,w)))}.
\end{eqnarray*}
This concludes the proof.
\end{proof}
\end{corollary}

\section{Concluding Remarks}

We have proven for certain large classes of pseudoconvex domains that the Bergman projection is bounded on weighted $L^p$ spaces where the weight $\sigma$ belongs to an appropriate $\mathcal{B}_p$ class. We have also obtained a partial converse, making an additional assumption on the kernel and integrability of the weight. To extend these results to much broader classes of domains using a similar approach, it is likely one would either have to obtain new estimates on the Bergman kernel (for example, on general (weakly) pseudoconvex domains of finite type) or adopt the approach in \cite{LS} used for domains of minimal regularity.

\noindent

\noindent \author{Zhenghui Huo\\
Department of Mathematics and Statistics\\
University of Toledo\\
{\tt  zhenghui.huo@utoledo.edu}\vspace{3mm}

\noindent \author{Nathan A. Wagner}\\
Department of Mathematics and Statistics\\
Washington University in St. Louis\\
{\tt nathanawagner@wustl.edu}

\vspace{0.5 cm}

\noindent \author{Brett D. Wick\\
Department of Mathematics and Statistics\\
Washington University in St. Louis\\
{\tt bwick@wustl.edu}\vspace{3mm}\\
\end{document}